\documentclass[reqno]{amsart}
\usepackage{amsfonts, amsmath, amsthm, amssymb, latexsym, graphicx}

\theoremstyle{plain}
\newtheorem{theorem}{Theorem}[section]
\newtheorem{corollary}[theorem]{Corollary}
\newtheorem{lemma}[theorem]{Lemma}
\newtheorem{proposition}[theorem]{Proposition}

\theoremstyle{definition}

\theoremstyle{remark}
\newtheorem{remark}[theorem]{Remark}

\numberwithin{equation}{section}

\title[Positive intertwiners for Bessel functions of type B]{Positive intertwiners 
 for Bessel functions of type B}
\author{Margit R\"osler}
\address{Institut f\"ur Mathematik, Universit\"at Paderborn, Warburger Str. 100, D-33098 Paderborn, Germany}
\email{roesler@math.upb.de}
\author{Michael Voit}
\address{Fakult\"at Mathematik, Technische Universit\"at Dortmund,
          Vogelpothsweg 87,
          D-44221 Dortmund, Germany}
\email{michael.voit@math.tu-dortmund.de}


\begin{document}
\date{\today}

\begin{abstract} 
Let $V_k$ denote Dunkl's intertwining operator for the root sytem $B_n$  with multiplicity 
 $k=(k_1,k_2)$ with $k_1\geq 0, k_2>0$. 
It was recently shown  that  the positivity of the operator $V_{k^\prime\!,k} = V_{k^\prime}\circ V_k^{-1}$ which intertwines
 the Dunkl operators associated with $k$ and
 $k^\prime=(k_1+h,k_2)$  implies that $h\in[k_2(n-1),\infty[\,\cup\,(\{0,k_2,\ldots,k_2(n-1)\}-\mathbb Z_+)$.
This is also a necessary condition for the existence of positive Sonine formulas between the associated Bessel functions.
In this paper we present two partial converse positive results:
For $k_1 \geq 0, \,k_2\in\{1/2,1,2\}$ and $h>k_2(n-1)$, the operator  $V_{k^\prime\!,k}$ is positive when restricted to functions which are invariant under the Weyl group, and there is an associated positive Sonine formula for the Bessel functions of type $B_n$. Moreover, the same positivity results hold for arbitrary $k_1\geq 0,  k_2>0$ and $h\in k_2\cdot \mathbb Z_+.$ 
 The proof is based on a formula of Baker and Forrester on connection coefficients between multivariate Laguerre polynomials 
and an approximation of Bessel functions by Laguerre polynomials.
\end{abstract}

\maketitle

\section{Introduction}

Let $R$ be a reduced root system in a 
 finite-dimensional Euclidean space 
 $(\frak a, \langle\,.\,,\,.\,\rangle)$ with  finite Coxeter group $W.$ Fix a  multiplicity function, i.e., a
  $W$-invariant function $k:R \to [0,\infty[$ and denote by $\{T_\xi(k), \xi\in \frak a\}$ the associated commuting family of
 rational Dunkl operators as introduced in \cite{D1}, see also \cite{D2, DX}.
Then there is a 
 unique isomorphism on the vector space $\mathbb C[\frak a]$  
  of polynomial functions on $\frak a$ which preserves the degree of homogeneity and satisfies 
$$ V_k(1)=1, \quad T_\xi(k) V_k = V_k \partial_\xi \quad \text{ for all } 
  \xi \in \frak a.$$
 By \cite{R1}, $V_k$ is  positive on $\mathbb C[\frak a],$ i.e. for $p\in\mathbb C[\frak a]$   with
 $p\geq 0$ we have $V_kp\geq 0$ on $\frak a$. This is equivalent to the fact that
 for each $x\in \frak a$ there is a unique compactly supported probability measure 
$\mu_{x}^k\in M^1(\frak a)$ such that the Dunkl kernel $E_k$ associated with $R$ and $k$ has the positive integral representation
\begin{equation}\label{Dunkl_kernel_int} E_k(x,z) = \int_{\frak a} e^{\langle \xi,z\rangle } d\mu_x^k(\xi), \quad \forall \, 
x \in \frak a, z\in \frak a_{\mathbb C}.
\end{equation}
The support of $\mu_x^k$ is contained in the convex hull  of 
the $W$-orbit of $x$. 
Eq.~\eqref{Dunkl_kernel_int} generalizes the Harish-Chandra integral representation of spherical functions  on symmetric spaces of Euclidean type (Ch.IV of \cite{Hel}), as
 for 
certain  multiplicities $k$, the Bessel functions 
\begin{equation}\label{Def-bessel-function}J_k(x,z) := \frac{1}{|W|}\sum_{w\in W} E_k(wx,z), \quad z\in \frak a_{\mathbb C},
\end{equation}
coincide with the spherical functions of a Cartan motion group, where $R$ and $k$ depend on the root space data of the 
underlying symmetric space, see \cite{O, dJ2} for details. The Bessel function $J_k$ is $W$-invariant in both arguments, and in the 
 geometric cases, the integral formula for $J_k$ obtained from
\eqref{Dunkl_kernel_int} by taking $W$-means is just the Harish-Chandra formula \cite[Prop. IV.4.8]{Hel}.

\medskip

We now consider two multiplicities $k, k^\prime$ on $R$ 
with  $k^\prime \geq k\geq 0$, i.e., 
$k^\prime(\alpha) \geq k(\alpha)\geq 0$ for all $\alpha \in R$.  
In \cite{RV3} we studied the operator
$$ V_{k^\prime\!,k} := V_{k^\prime} \circ V_k^{-1} $$
which
 intertwines the 
Dunkl operators  with multiplicities $k$ and $k^\prime$:
$$ T_\xi(k^\prime) V_{k^\prime\!, k} = V_{k^\prime\!, k} T_\xi(k) \quad (\xi \in \frak a).$$
It had been conjectured until recently that
 $V_{k^\prime\!, k}$ is also positive. This is equivalent 
 to the statement that for each $x\in \frak a$, there is
a unique compactly supported probability measure $\mu_x^{k^\prime\!,k}\in M^1(\frak a)$ such that the Sonine formula
\begin{equation}\label{Sonine_E}  E_{k^\prime}(x,z) = \int_{\frak a} E_k(\xi, z)\, d\mu_x^{k^\prime\!,k}(\xi)
\quad \text{for all } z\in \frak a_{\mathbb C}
\end{equation}
holds. 
Notice  that $\eqref{Sonine_E}$ yields a corresponding formula for the Bessel function:
\begin{equation}\label{Sonine_B}  J_{k^\prime}(x,z) = \int_{\frak a} J_k(\xi,z) \,d\widetilde \mu_x^{k^\prime\!,k}(\xi) 
\quad (z\in \frak a_{\mathbb C})
 \end{equation}
with some $W$-invariant probability measure $\widetilde \mu_x^{k^\prime\!,k}.$ Denote by $\widetilde V_{k^\prime\!, k}$ the restriction of $V_{k^\prime\!, k}$ to $W$-invariant functions. As in \cite{RV3}, it is easy to see that 
the existence of a positive Sonine formula \eqref{Sonine_B} for the Bessel functions is equivalent to the positivity of $\widetilde V_{k^\prime\!, k}$.
\medskip

In the rank-one case with 
$R= \{\pm 1\}\subset \mathbb R$, we have
$\, J_k(x,y) = j_{k-1/2}(ixy) \,$
with the one-dimensional Bessel function
\begin{equation}\label{B_eindim} j_\alpha(z)  = \, _0F_1(\alpha+1;-z^2/4) \quad\quad 
(\alpha\in\mathbb C\setminus\{-1,-2,\ldots\}).\end{equation}
Here \eqref{Sonine_B} is just the well-known classical 
Sonine formula  (see e.g.~\cite{A}):
\begin{equation}\label{intrep-1-dim}
j_{\alpha+\beta}(z)= 2\frac{\Gamma(\alpha+\beta+1)}{\Gamma(\alpha+1)\Gamma(\beta)}
\int_{0}^1j_{\alpha}(zx)  x^{2\alpha+1}(1-x^2)^{\beta-1} dx
\end{equation}
for  $\alpha\in]-1,\infty[$ and $\beta \in ]0,\infty[$.
It was proven in \cite{X} that the operator $V_{k^\prime\!, k}$ with $k^\prime > k \geq 0$  is also  positive in the rank-one setting.

On the other hand, for the root systems
$$ B_n = \{ \pm e_i, \, \pm e_i \pm e_j, \, 1\leq i < j \leq n\} \subset \mathbb R^n$$
with $n\ge2$ and certain parameters $k^\prime \geq k \geq 0, $
\eqref{Sonine_E} and  \eqref{Sonine_B} were disproved 
in \cite{RV3}. For these root systems we write the  multiplicities as $k=(k_1,k_2)$ with $k_1, k_2$ the values of $k$ on
 the roots $\,\pm e_i$ and $\,\pm e_i\pm e_j,$
respectively. The following result is shown in Section 3 of \cite{RV3} for the Bessel functions $J_{k}^B$ of type $B_n$.

\begin{theorem}\label{main_Bessel} Let $k=(k_1, k_2)$ with  $k_1 \geq 0, \> k_2 >0$, and consider $k^\prime = (k_1+h,k_2)$ with
  $h>0$. Assume that there exists
 a probability measure $m\in M_b(\mathbb R^n)$ such that the restricted Sonine formula
\begin{equation}\label{Sonine_1} J_{k^\prime}^B({\bf 1},iy) = \int_{\mathbb R^n}  J_{k}^B(\xi,iy)\> dm(\xi)
\quad\text{for all }y \in \mathbb R^n
\end{equation}
holds, where ${\bf 1} = (1, \ldots, 1) \in \mathbb R^n$. Then 
$h$ is contained in the set
$$ \Sigma(k_2):= \,]k_2(n-1), \infty [ \, \cup \, \bigl(\{0, k_2, \ldots, k_2(n-1)\} -\mathbb Z_+\bigr).$$
\end{theorem}

Therefore, positivity of $\widetilde V_{k^\prime\!, k}$ or $V_{k^\prime\!, k}$ requires that  $h\in\Sigma(k_2)$.  Theorem \ref{main_Bessel}
 is related  to a classical result of Gindikin \cite{G} about  the Wallach set which parametrizes those Riesz distributions on a Euclidean Jordan algebra which are actually positive measures, 
 see \cite{FK}.
 Related  results for Riesz and Beta
 distributions  on symmetric cones and in the Dunkl setting are given in \cite{RV2, R3}.

On the other hand, there are some positive Sonine formulas for $J_{k^\prime}^B$ in terms of $J_k^B$  with $k^\prime=(k_1+h,k_2)$ if $k_1 \geq 0$ and $h$ is large enough; these will be treated in Section \ref{large}. 
 Namely, if $k_2\in\{1/2,1,2\}$ and $h>k_2(n-1)$, then  a Sonine formula 
 follows from the fact that in this case
 the Bessel functions $J_{k}^B$ are closely related to Bessel functions 
on matrix cones over $\mathbb R, \mathbb C, \mathbb H$ 
where  positive Sonine formulas are available by \cite{H, RV2}. Moreover, a restricted explicit Sonine formula of the form \eqref{Sonine_1}
was obtained in \cite{RV3} for arbitrary $k_1\geq 0, k_2 >0$  and $h>k_2(n-1)$ as a consequence of Kadell's \cite{Ka} generalization of the Selberg integral.

The main results of this paper are contained in Sections \ref{discrete} and \ref{infty}. In particular, in Theorem \ref{existence-pos-int-rep-lag} we prove that 
 the Bessel function $J_{k}^B$ satisfies a Sonine formula for all   $k_1, k_2>0$ and $\,h\in k_2\cdot\mathbb Z_+$, which includes all parameters $h$ in the discrete part of the generalized Wallach set 
 $$\{0,k_2,\ldots, k_2(n-1)\}\, \cup \, [k_2(n-1), \infty[\,.$$ 
 The proof of this result is based on a formula of Baker and Forrester \cite{BF1} on positive connection coefficients between multivariate Laguerre polynomial systems 
and an approximation of Bessel functions by Laguerre polynomials.  Finally in Section \ref{infty}, we consider the limit $h\to\infty$
 under which the Bessel functions of type B tend to those of type A. This leads to a positive integral representation of Bessel functions of type $A$ in terms of such of type $B$.

\section{Basic facts on  Bessel functions}

We start with some background on  rational Dunkl theory  from  \cite{D1, D2, DJO, R1}. 
Let  $R$ be  a reduced root system in a finite-dimensional Euclidean space $(\frak a, \langle\,.\,,\,.\,\rangle)$  and  
$W$ 
the associated finite Coxeter group.
The Dunkl operators
 associated with $R$ and multiplicity  $k$ are 
 $$ T_\xi(k) = \partial_\xi  + \frac{1}{2}\sum_{\alpha\in R}
 k(\alpha) \langle \alpha, \xi\rangle \frac{1}{\langle \alpha,\,.\,\rangle}(1-\sigma_\alpha) , \quad \xi \in \frak a
 $$
 where the action of $W$ on functions $f: \frak a \to \mathbb C$  is given by $w.f(x) = f(w^{-1}x).$
 It was shown in \cite{D1} that the  $T_\xi(k), \, \xi \in \frak a$,   commute. 
 Multiplicities $k\ge0$ are regular, i.e., 
the joint kernel of the $T_\xi(k)$, considered as linear operators on 
 polynomials, consists of the constants only.  This is equivalent to the existence of a necessarily unique
 intertwining operator 
 $V_k$ as described in the introduction; see \cite{DJO}. 
 Moreover, for each  $y \in \frak a_\mathbb C,$ there is a unique solution $f = E_k(\,.\,,y)$ of the 
joint eigenvalue problem
$$ T_\xi(k)f = \langle \xi,y\rangle f\quad \forall \, \xi \in \frak a, \,\, f(0)=1.$$
The function $E_k$ is called the Dunkl kernel. The mapping $(x,y)\mapsto E_k(x,y)$ is
 analytic on $ \frak a_\mathbb C\times \frak a_\mathbb C$ with $ E_k(x,y) = E_k(y,x)$, $ E_k(x,0)=1$,  and
 $$   E_k(\lambda x,y) = E_k(x, \lambda y), \quad E_k(wx, wy) = E_k(x,y)\quad (\lambda \in \mathbb C, w \in W).$$ 

In this paper, we mainly consider the Dunkl kernel $E_k^B$ and the Bessel function $J_k^B$ associated with the root system 
$B_n$ on  $\mathbb R^n$ with its usual inner product, c.f. \eqref{Def-bessel-function}.
The associated reflection group is the hyperoctahedral group $W(B_n)= S_n \ltimes \mathbb Z_2^n, $ and
the multiplicity has the form $k=(k_1, k_2)$ as explained in the introduction. In Section \ref{infty}, we shall also consider the Bessel function $J_k^A$ associated with the root system $$A_{n-1} = \{\pm(e_i-e_j), 1 \leq i < j \leq n\}\subset \mathbb R^n.$$ Here the multiplicity is given by a single parameter $k\in [0, \infty[.$ 
The Bessel functions $J_k^B$ and $J_k^A$ have simple expressions in terms of multivariate hypergeometric functions in the sense of \cite{K, M}. To recall these we need some more notation.
Let 
$$\Lambda_n^+ =\{ \lambda\in \mathbb Z_+^n: \lambda_1 \geq \cdots \geq \lambda_n\}$$ 
be the set of partitions of length at most $n$. We denote  
by $C_\lambda^\alpha, \, \lambda \in \Lambda_n^+$  the Jack polynomials of index $\alpha>0$ in $n$ variables (see \cite{Sta}),
normalized such that
\begin{equation}\label{normal}
(z_1 + \cdots + z_n)^m = \sum_{|\lambda|=m} C_\lambda^\alpha(z)\quad (m \in \mathbb Z_+).
\end{equation}

Following \cite{K}, we define for $\mu \in \mathbb C$ with 
$\text{Re}\, \mu > \frac{1}{\alpha}(n-1)$ and arguments $z,w \in \mathbb C^n$ the hypergeometric function 

\begin{equation}\label{def-hypergeometric-function} 
    _0F_1^\alpha(\mu; z,w):= \sum_{\lambda\in \Lambda_n^+} \frac{1}{[\mu]_\lambda^\alpha\, |\lambda|!} \cdot \frac{C_\lambda^\alpha(z) 
C_\lambda^\alpha(w)}{C_\lambda^\alpha({\bf 1})}, \quad {\bf 1} = (1, \ldots, 1) \in \mathbb R^n,\end{equation}
with 
 the generalized Pochhammer symbol 
\begin{equation}\label{general_Pochhammer}
 [\mu]_\lambda^\alpha = \prod_{j=1}^n \bigl(\mu - \frac{1}{\alpha} (j-1)\bigr)_{\lambda_j}.
\end{equation}
If  $n=1$, then the $C_\lambda^\alpha$ are independent of $\alpha$ and given by  $C_\lambda^\alpha(z) = z^\lambda, \, \lambda\in \mathbb Z_+$, and 
$$ _0F_1^\alpha(\mu;-\frac{z^2}{4},1) = j_{\mu-1}(z).$$
Let $k=(k_1,k_2)$ on $B_n$ with $k_1\geq 0$ and $k_2>0$. Then according to \cite[Prop. 4.5]{R2} (c.f. also \cite[Sect. 6]{BF1}), 
\begin{equation}\label{ident_Bessel_B} J_k^B(z,w) = \,_0F_1^\alpha\Bigl(\mu; \frac{z^2}{2}, \frac{w^2}{2}\Bigr)
	\end{equation}
with $\, \alpha = \frac{1}{k_2}\, $ and $\, \displaystyle  \mu=\mu(k)= k_1+k_2(n-1)+ \frac{1}{2}.$ 

Moreover,  for $n \geq 2$ and $k\in \,]0,\infty[,$ the Bessel function of type $A_{n-1}$ has the following hypergeometric expansion (\cite[Sect. 3]{BF2}):
 \begin{equation}\label{Bessel_A} J_k^A(z,w) = \,_0F_0^\alpha(z,w) = \sum_{\lambda \in \Lambda_n^+} \frac{1}{|\lambda|!}\cdot 
\frac{C_\lambda^\alpha(z) C_\lambda^\alpha(w)}{C_\lambda^\alpha({\bf
    1})}
\quad \text{with }\, \alpha =1/k\,. \end{equation}

\section{Explicit Sonine formulas for large parameters}\label{large}

We first recapitulate a restricted Sonine formula of the form \eqref{Sonine_1} for  $J_k^B$ from \cite{RV3}.
 Let $k = (k_1, k_2)$ with $\,k_1 \geq 0$ and $k_2>0.$ 
For a real parameter  $h> k_2(n-1)$ consider the probability density 
$$ f_{k,h}(x) := c_{k,h}^{-1}\cdot \prod_{i=1}^n (x_i^2)^{k_1} (1-x_i^2)^{h-k_2(n-1)-1} 
\prod_{i<j} |x_i^2-x_j^2|^{2k_2}$$
on $[0,1]^n$ with the normalization constant (a Selberg-type integral)
$$ c_{k,h} = \int_{[0,1]^n} \prod_{i=1}^n (x_i^2)^{k_1} (1-x_i^2)^{h-k_2(n-1)-1} 
\prod_{i<j} |x_i^2-x_j^2|^{2k_2}dx.$$ 
Then according to Eq. (3.4) in  \cite{RV3},
\begin{equation}\label{density} J_{(k_1+h, k_2)}^B({\bf 1},z) = \int_{[0,1]^n} J_k^B(x,z) f_{k,h}(x)dx\quad (z \in \mathbb C^n). \end{equation}
The proof of this formula  is based on \eqref{ident_Bessel_B} and Kadell's \cite{Ka} generalization of the Selberg integral, which implies a restricted Sonine formula for $\,_0F_1^\alpha.$
 The restricted Sonine formula 
 \eqref{density} can be extended to a complete  Sonine formula, i.e. a Sonine formula for Bessel functions with arbitrary second argument, in the cases $k_2=1/2,1,2\,$ via
 Bessel functions on matrix cones.
 To explain this, we recapitulate some notations from \cite{FK, R2, RV2}.
For the (skew) fields $\mathbb F:= \mathbb R, \mathbb C, \mathbb H$
 with real dimension $d=1,2,4$  consider the vector spaces $H_n:=H_n(\mathbb F)$ of all Hermitian 
$n\times n$ matrices over  $\mathbb F$ as well as the 
cones
$\Pi_n:=\Pi_n(\mathbb F)\subset H_n$
 of all positive semidefinite matrices. Let $\Omega_n$ be the interior of $\Pi_n$ consisting of all strictly positive
definite matrices.
The Bessel functions associated with the symmetric cone $\Omega_n$ (in the sense of \cite{FK}) are defined in terms of the spherical polynomials
\[ \Phi_\lambda (a) = \int_{U_n} \Delta_\lambda(uau^{-1})du \quad (\lambda \in \Lambda_n^+, \, \, a\in H_n).\]
Here
$du$ is the normalized Haar measure of the compact group $\,U_n=U_n(\mathbb F)$ and $\Delta_\lambda$ denotes the power function
\[ \Delta_\lambda(a) := \Delta_1(a)^{\lambda_1-\lambda_2}
\Delta_2(a)^{\lambda_2-\lambda_3} \cdot\ldots\cdot \Delta_q(a)^{\lambda_q}\]
 with the principal minors $\Delta_i(a)$. 
We renormalize the spherical polynomials according to 
$\, Z_\lambda := c_\lambda \Phi_\lambda$
with  $c_\lambda >0$ 
such that
$$(tr \,a)^k \,=\, \sum_{|\lambda|=k} Z_\lambda(a)
\quad\quad\text{for}\quad
 k\in \mathbb N_0\,;$$
c.f. Section XI.5.~of \cite {FK}.
The $Z_\lambda$ 
 depend only on the eigenvalues of their argument and are given in terms of Jack polynomials: for $a \in H_n$ with eigenvalues $x = (x_1, \ldots, x_n) \in \mathbb R^n$, 
$$Z_\lambda(a) =
C_\lambda^\alpha(x) \quad \text{with } \alpha = \frac{2}{d}\,,$$
see  \cite {FK, M2}. Let
$$ \mu_0:= \frac{d}{2}(n-1).$$
Then for $\mu > \mu_0$, the Bessel function $J_\mu$ associated with $\Omega_n$ 
is defined according to  \cite{FK} as the $\, _0F_1$-hypergeometric series \begin{equation}\label{power-j}
 J_\mu(a) =
 \sum_{\lambda\in \Lambda_n^+} \frac{(-1)^{|\lambda|}}{[\mu]_\lambda^{2/d}\,|\lambda|!}\,
Z_\lambda(a) \quad (a\in H_n).
\end{equation}
Notice that here
$[\mu]_\lambda^{2/d} \not=
0$ for
all  $\lambda$. Comparing formulas \eqref{power-j} and \eqref{ident_Bessel_B}, one obtains 
 for $a\in H_n$ with eigenvalues $x \in \mathbb R^n$ and $\mu\geq \mu_0 + \frac{1}{2}$ the identity 
 \begin{equation}\label{D_identity} J_\mu(a^2) = J_k^B(2ix,{\bf 1}) \quad \text{ with } \,k=k(\mu,d) = 
\bigl(\mu - \mu_0 - \frac{1}{2}, \frac{d}{2}\bigr); \end{equation}
see \cite[Cor. 4.4]{R2}. Recall that $J_k^B$ is invariant under $W(B_n)$ in both variables. We denote the  associated standard Weyl chamber
by 
 $$ C_n^B:= \{ x \in \mathbb R^n: x_1 \geq \ldots \geq x_n\geq 0\}.$$
Now consider $r,s\in \Pi_n$ with ordered eigenvalues $x=(x_1,\ldots, x_n), \, y = (y_1, \ldots, y_n) \in C_n^B$. Then  
 by Corollary 4.6. of \cite{R2},
\begin{equation}\label{D_int} J_k^B(2ix,y)=\int_{U_n} J_\mu(r us^2u^{-1}r)\, du \quad \text{ with } k=k(\mu,d).\end{equation}
For indices  $\mu, \nu > \mu_0$ 
we introduce the Beta-Riesz probability measures
\begin{equation}\label{def-beta-dist}
d\beta_{\mu,\nu}(r):= \frac{1}{ B_\Omega(\mu,\nu)}  \Delta(r)^{\mu-1-\mu_0}\Delta(I-r)^{\nu-1-\mu_0}
dr\Bigl|_{\Omega_n^I} \end{equation}
on the relatively compact set $\Omega_n^I:=\{r\in \Omega_n: \> I_n-r\in  \Omega_n\},$ where
 $$ B_\Omega(\mu,\nu):=\int_{\Omega_n^I} \Delta(r)^{\mu-1-\mu_0}\Delta(I-r)^{\nu-1-\mu_0}
dr. $$
 We recall the following Sonine formula from
 \cite[Theorem 1]{RV2}; see also (2.6') of \cite{H} for $\mathbb F=\mathbb R$.

\begin{proposition}\label{int-rep-classical}
For all $\mu,\nu >\mu_0$  and $r\in\Pi_n$,
$$J_{\mu+\nu}(r)=
\int_{ \Omega_n^I } J_\mu(r s) d\beta_{\mu,\nu}(s).$$
\end{proposition}

Here we adopt a common notation in the literature and consider the Bessel function $J_\mu$ as a function of the eigenvalues of the (not necessarily Hermitian) matrix $rs$, which 
coincide with those of the Hermitian matrices $\sqrt s \,r\sqrt s$ and 
$\sqrt r \,s\sqrt r.$

From Proposition \ref{int-rep-classical} and formula \eqref{D_int} we now derive a Sonine representation for $J_k^B$. 
For this, consider
the map $\,\sigma:\Pi_n\to C_n^B$ where $\sigma(r)$ is the ordered spectrum of $r\in \Pi_n$. We shall identify $x\in \mathbb R^n$ with the diagonal matrix $\text{diag}(x_1, \ldots, x_n)\in H_n$. 
 Theorem VI.2.3 of \cite{FK} states that 
\begin{equation}\label{polar} \int_{\Pi_n} f(r) \,dr\,=\,c_0 \int_{U_n}\int_{ C_n^B} f(u x u^{-1})\cdot 
\prod_{i<j} (x_i-x_j)^d \,dx\, du \end{equation}
for integrable functions $f:\Pi_n\to\mathbb C, $ with some constant $c_0>0$. Let 
$$C_{n,1}^B:=\{x\in C_n^B:\> x_1\le1\} = \{ x \in \mathbb R^n: 1 \geq x_1 \geq \cdots \geq x_n \geq 0\}.$$ 
Then by \eqref{polar}, the image measure of $\beta_{\mu,\nu}$ under $\sigma$
is  the probability measure
\begin{equation}
d\rho_{k_1,k_2,h}(x)= \frac{1}{B_{k_1,k_2,h}}  \prod_{i=1}^n x_i^{k_1-1/2}(1-x_i)^{h-1-\mu_0} \cdot \prod_{i<j}(x_i-x_j)^{2k_2}
dt\Bigl|_{C_{n,1}^B} \end{equation}
on $C_{n,1}^B$
with the
normalization constant $B_{k_1,k_2,h}=B_\Omega(\mu,\nu)/c_0$ where $(k_1,k_2)=k(\mu,d)$ as in \eqref{D_identity} and $h = \nu$. We obtain

\begin{theorem}\label{int-rep-classical-chamber} Let $\,\mathbb F \in \{\mathbb R, \mathbb C, \mathbb H\}$ and $k_2 = d/2$ with $d= \text{dim}_{\mathbb R} \mathbb F.$
Let further $k_1\ge 0$ and $h> \mu_0=\frac{d}{2}(n-1)$.
Then for all $z\in \mathbb C^n$ and $x\in C_n^B$,
$$ J_{(k_1+h,k_2)}^B(x,z) = \int_{C_{n,1}^B} \int_{U_n(\mathbb F)} J_{(k_1,k_2)}^B
\bigl(\sqrt{\sigma(x u\xi u^{-1}x)}\,,z\bigr)\,du\,
   d\rho_{k_1,k_2,h}(\xi).$$
\end{theorem}

\begin{proof}  Let $x,y\in  C_n^B$ and put $\mu=k_1+\mu_0+\frac{1}{2}.$ Formulas \eqref{D_int}, \eqref{polar}  and Proposition \ref{int-rep-classical} imply that
\begin{align*}
J_{(k_1+h,k_2)}^B(x,2iy) \,&
=\,  \int_{U_n}\int_{\Omega_n^I} J_{\mu}\bigl(
x u y^2 u^{-1}xs\bigr)\, d\beta_{\mu,h}(s)\,du \\ &= \,
\int_{U_n}\int_{\Omega_n^I} J_{\mu}\bigl(
 u y^2 u^{-1}x sx \bigr)\, d\beta_{\mu,h}(s) \,du\\
 &= \,
\int_{U_n}\int_{U_n}\int_{C_{n,1}^B} J_\mu\bigl(uy^2u^{-1}\cdot  x v\xi v^{-1}x\bigr) \,d\rho_{k_1,k_2, h}(\xi)\,dv\, du \\
&=\,\int_{C_{n,1}^B} \int_{U_n} J_{(k_1,k_2)}^B
\bigl(\sqrt{\sigma(x v  \xi v^{-1}x)}\,,2iy\bigr)\,dv\,
   d\rho_{k_1,k_2,h}(\xi).
   \end{align*}
 Analytic extension in the first argument now yields the assertion.
\end{proof}

Note that for  $k_2 =1/2,1,2$ and $\eta={\bf 1}$, the integral representation in Theorem \ref{int-rep-classical-chamber} 
coincides with formula \eqref{density}.

\section{Positive Sonine formulas for discrete parameters}\label{discrete}
 
Besides the  Sonine formulas in Theorem \ref{int-rep-classical-chamber} and Eq. \eqref{density}, 
Theorem \ref{main_Bessel} admits the following  partial converse statement:

\begin{theorem}\label{existence-pos-int-rep-lag} Let $k=(k_1,k_2)$ with $k_1\geq 0, k_2 >0$ and $\,h\in k_2\cdot\mathbb Z_+ =\{0, k_2, 2k_2,\ldots\}$.
 Then
  for each $x\in C_n^B$ there exists a unique probability measure $m_x\in M^1(C_n^B)$ such that
\begin{equation}\label{Sonine_1-allg} J_{(k_1+h,k_2)}^B(x,z) =
 \int_{C_n^B}  J_{k}^B(\xi,z) \,dm_x(\xi)\end{equation}
for all $z\in\mathbb C^n.$ The support of $m_x$ is contained in 
$[0,x]:= [0,x_1]\times\ldots\times [0,x_n].$
\end{theorem}

The uniqueness statement in  Theorem \ref{existence-pos-int-rep-lag} follows from the injectivity of the Dunkl transform of bounded measures, see \cite[Theorem 2.6]{RV0}.
For the proof of the existence part, it suffices to consider $ z= iy$ with $y \in \mathbb R^n$ and prove the case $h=k_2$. 
In this case, the theorem can be derived from an explicit formula of Baker and Forrester \cite[formula (4.23)]{BF1}
on the connection coefficients between multivariate Laguerre polynomial systems. To start with, recall that according to \cite[Proposition 4.3]{BF1}, 
the Laguerre polynomials in $n$ variables associated with parameters $\alpha >0$ and $a > -1$ are given by 
\begin{equation}\label{def-Laguerre} 
L_\kappa^{a}(x; \alpha)=\frac{[a+q]_\kappa^{\alpha}}{|\kappa|!}
\sum_{\lambda\subseteq \kappa} \binom{\kappa}{\lambda}
\frac{(-1)^{|\lambda|}}{[a+q]_\lambda^{\alpha}}\,\frac{C_\lambda^\alpha(x)}{ C_\lambda^\alpha({\bf 1}) }\quad (\kappa \in \Lambda^+_n\,, \, \, x\in \mathbb R^n)
\end{equation}
where $ q = 1 + (n-1)/\alpha,\,$ the notion $\lambda\subseteq \kappa$ means that the diagram of $\lambda$ is contained in that of $\kappa$, i.e. 
$\lambda_i \leq \kappa_i$ for all $i=1, \ldots, n,$ and the generalized binomial coefficients $\,\binom{\kappa}{\lambda}=\binom{\kappa}{\lambda}_\alpha $ are defined by 
the binomial formula for the Jack polynomials, 
$$ \frac{C_\kappa^\alpha (x +{\bf 1})}{C_\kappa^\alpha({\bf 1})} = \sum_{\lambda \subseteq \kappa}\binom{\kappa}{\lambda}\,
\frac{C_\lambda^\alpha(x)}{C_\lambda^\alpha({\bf 1})}.$$
We consider the renormalized Laguerre polynomials
$$ \widetilde L_\kappa^a(x;\alpha) := \frac{L_\kappa^a(x;\alpha)}{L_\kappa^a(0;\alpha)} = \sum_{\lambda\subseteq \kappa} 
\binom{\kappa}{\lambda}\,
\frac{(-1)^{|\lambda|}}{[a+q]_\lambda^{\alpha}}\,\frac{C_\lambda^\alpha(x)}{ C_\lambda^\alpha({\bf 1}) }$$
and choose the parameters according to those in 
\eqref{ident_Bessel_B}, namely $\,\alpha:= \frac{1}{k_2},$ $\, a := k_1 -\frac{1}{2}.$ 
Then $\, a+q = k_1 + k_2(n-1) + \frac{1}{2} = \mu(k) =: \mu.$ 

\smallskip

The first step towards the proof of Theorem \ref{existence-pos-int-rep-lag}
is the 
 following limit result for  Laguerre polynomials which generalizes  Proposition 3.3 of \cite{Far}
for $\alpha = 1.$ In the one-dimensional case $n=1$, this limit transition is well-known.

\begin{lemma}\label{laguerre-to-bessel}
 Let $k= (k_1,k_2)$ with $k_1\geq 0, k_2 >0, $ put $ \, \alpha:= \frac{1}{k_2}\,, \,a:= k_1-\frac{1}{2}$ and fix $\,y\in \mathbb R^n$. For $x \in C_n^B\,$ 
 consider the sequence of partitions $(\lambda_j(x))_{j\in \mathbb N} \subseteq \Lambda^+_n$ with 
$\lambda_j(x):= \lfloor j\cdot x\rfloor$ where Gaussian brackets are taken componentwise.
 Then
 $\lim_{j\to\infty} \lambda_j(x)/j=x$, and
$$\lim_{j\to\infty}\widetilde L_{\lambda_j(x)}^{a}(y^2/j;\alpha) =\,J_k^B(iy,2\sqrt{x}\,)$$
locally  uniformly in $x$ (here the square root is also taken componentwise).  
\end{lemma}

\begin{proof} We adopt the method from \cite{Far}. First, recall the monic renormalization of the Jack polynomials 
$C_\lambda^\alpha$, which is  determined by 
$$ P_\lambda^\alpha = m_\lambda + \sum_{\mu < \lambda} d_{\lambda\mu}(\alpha) m_\mu $$
where the $m_\lambda$ are the monomial symmetric functions $\, m_\lambda(z) = \sum_{w\in S_n} z^{w.\lambda} $ and $\mu < \nu$ means that 
$\mu$ is strictly lower than $\nu$ in the usual dominance order on $\Lambda_n^+$. 
According to \cite[formula (16)]{K} and the identities on p. 15 of \cite{KS}, 
 $$ P_\lambda^\alpha = \frac{c_\lambda^\prime}{\alpha^{|\lambda|}|\lambda|!}\, C_\lambda^\alpha \quad\text{with} \quad
  c_\lambda^\prime = \prod_{s\in \lambda} (\alpha(a_\lambda(s) +1) + l_\lambda(s)),$$
 where $a_\lambda(s)$ and $l_\lambda(s)$ denote the arm-length and leg-length of $s\in \lambda$, c.f. \cite{KS, Sah} for the definitions.  
 According to \cite{OO}, the  generalized binomial coefficients $ \binom{\kappa}{\lambda}$  can be written in terms of the so-called shifted 
 Jack polynomials 
 $P_\lambda^*(x,1/\alpha)$ as
 $$ \binom{\kappa}{\lambda} = \frac{P_\lambda^*(\kappa;1/\alpha)}{P_\lambda^*(\lambda;1/\alpha)}.$$
 Moreover, 
 $$ P_\lambda^*(\lambda; 1/\alpha) = \frac{1}{\alpha^{|\lambda|}} c_\lambda^\prime.$$
The shifted Jack polynomials are of the form (c.f. \cite{OO})
$$ P_\lambda^*(z;1/\alpha) = P_\lambda^\alpha (z) + \text{terms of lower degree in dominance order.} $$
Therefore
\begin{equation}\label{limitbinom} \lim_{j\to\infty} \frac{1}{j^{|\lambda|}}\binom{\lambda_j(x)}{\lambda}  = \,\frac{\alpha^{|\lambda|}}{c_\lambda^\prime} \cdot
\lim_{j\to \infty} P_\lambda^\alpha\bigl(\frac{\lambda_j(x)}{j}\bigr) = \frac{\alpha^{|\lambda|}}{c_\lambda^\prime} \cdot 
P_\lambda^\alpha(x) = \frac{C_\lambda^\alpha(x)}{|\lambda|!},
 \end{equation}
and the convergence is locally uniform in $x$. 
With $\alpha = \frac{1}{k_2}, \,  a= k_1-\frac{1}{2}$ and $\mu=\mu(k)= k_1+k_2(n-1) + \frac{1}{2}\,$ we thus obtain
\begin{align} \label{limit_laguerre}\lim_{j\to \infty} \widetilde L^a_{\lambda_j(x)} \bigl(y^2/j;\alpha\bigr) &= 
\lim_{j\to\infty} \sum_{\lambda\subseteq \lambda_j(x)}  \frac{1}{j^{|\lambda|}}\binom{\lambda_j(x)}{\lambda} 
\frac{(-1)^{|\lambda|}}{[\mu]_\lambda^\alpha} \frac{C_\lambda^\alpha(y^2)}{C_\lambda^\alpha({\bf 1})} \notag\\
& = \sum_{\lambda\in \Lambda_n^+} 
\frac{(-1)^{|\lambda|}}{[\mu]_\lambda^\alpha} \cdot \frac{C_\lambda^\alpha(x)\,C_\lambda^\alpha(y^2)}{|\lambda|!\, 
C_\lambda^\alpha({\bf 1})}\, = \, J_k^B(iy, 2\sqrt{x}).
\end{align}   
For the uniformity statement, we proceed as in \cite{V}. We use that the binomial coefficients satisfy 
$ \binom{\kappa}{\lambda} \geq 0 $ for $\lambda \subseteq \kappa$ (Theorem 5 of \cite{S2}) as well as 
$$ \sum_{|\lambda|=m} \binom{\kappa}{\lambda} = \binom{|\kappa|}{m}, \quad m \in \mathbb Z_+ $$
from \cite{L1}. Note also that $\, [\mu]_\lambda^\alpha \geq \, (\frac{1}{2})^{|\lambda|}.$ As the coefficients of the $C_\lambda^\alpha$ in their monomial expansion are nonnegative, we may estimate
$$ \Big\vert\frac{C_\lambda^\alpha(y^2)}{C_\lambda^\alpha({\bf 1})}\Big\vert \leq \|y\|_\infty^{\,2|\lambda|}\,.$$
We therefore obtain
\begin{align*} \sum_{\lambda\subseteq \lambda_j(x)}  \frac{1}{j^{|\lambda|}}\binom{\lambda_j(x)}{\lambda} &
\frac{1}{[\mu]_\lambda^\alpha}\, \Big\vert\frac{C_\lambda^\alpha(y^2)}{C_\lambda^\alpha({\bf 1})}\Big\vert \, \leq \,
 \sum_{m=0}^\infty \,\frac{2^m}{j^m} \,\|y\|_\infty^{2m}\cdot \sum_{|\lambda|=m} \binom{\lambda_j(x)}{\lambda}\, \notag\\ 
&=\,\sum_{m=0}^\infty \Bigl(\frac{2\|y\|_\infty^2}{j}\Bigr)^{\!m} \cdot \!\binom{|\lambda_j(x)|}{m}\, \leq \,\sum_{m=0}^\infty 
\frac{(2n\|y\|_\infty^2\|x\|_\infty)^m}{m!}\,.
\end{align*}
This estimate shows that the convergence in \eqref{limit_laguerre} is locally uniform in $x\in C_n^B.$ 
\end{proof}

The following observation is the central ingredient for the proof of
Theorem \ref{existence-pos-int-rep-lag}. It follows immediately from   identity (4.23) of \cite{BF1}.

\begin{lemma}\label{pos-connection-Lagu} 
For each partition $\kappa\in \Lambda^+_n$ there exist connection coefficients $c_{\kappa,\lambda}= c_{\kappa,\lambda}(a,\alpha)\ge0$ with
$\,\sum_{\lambda\subseteq\kappa}c_{\kappa,\lambda}=1\,$
such that
$$\widetilde L_{\kappa}^{a+ 1/\alpha}(x; \alpha) \,= \,
\sum_{\lambda\subseteq\kappa}c_{\kappa,\lambda}\,\widetilde L_{\lambda}^{a}(x;\alpha)  \quad\quad(x\in \mathbb R^n) .$$
\end{lemma}

 \begin{proof}[Proof of Theorem \ref{existence-pos-int-rep-lag}]
Let $k=(k_1,k_2), \,a:= k_1-\frac{1}{2}$ and  $\alpha = \frac{1}{k_2}.$  Fix $y\in \mathbb R^n.$ For 
$x\in C_n^B\,$  consider the sequence of partitions 
$(\lambda_j= \lambda_j(x))_{j\in \mathbb N}$
   as described in Lemma \ref{laguerre-to-bessel} with
 $\lim_{j\to\infty} \lambda_j/j=x$. 
 We 
 proceed  as in \cite{RV3} and introduce the discrete probability measures
 $$  \mu_j:= \mu_j(x):=\sum_{\lambda \subseteq \lambda_j}c_{\lambda_{j},\lambda}\,\delta_{\lambda/j}
 \in M^1(\mathbb R^n), \quad j \in \mathbb N;$$
 here $\delta_\xi$ denotes the point measure in $\xi\in \mathbb R^n.$ 
The supports of the $\mu_j$  are contained in the compact set $\,[0, x] \cap C_n^B\,.$ 
In terms of these measures,   Lemma \ref{pos-connection-Lagu} implies that
  \[
\widetilde L_{\lambda_{j}}^{a+k_2}(\xi;\alpha)\, = \, \int_{[0,\eta]} 
\widetilde L_{jw}^{a}(\xi;\alpha) d\mu_j(w), \quad \xi \in \mathbb R^n.
 \]
By Prohorov's theorem (see e.g.~\cite{Kal}),
  the set $\{\mu_j:\> j\in  \mathbb N\}$ is relatively sequentially compact. 
  After passing to a subsequence if necessary, we obtain that the $\mu_j$ tend weakly to a probability measure 
  $m_x\in M^1([0,x])$  as $j\to \infty.$ 
Using Lemma \ref{laguerre-to-bessel}, we conclude that
 \begin{align*}
 J_{(k_1+k_2, k_2)}^B\bigl(iy, 2\sqrt{x}\,\bigr) &=\lim_{j\to\infty}\widetilde L_{\lambda_{j}}^{a+k_2}(y^2/j;\alpha)\,  =\,
  \lim_{j\to\infty}
 \int_{[0,x]} \widetilde L_{jw}^a(y^2/j;\alpha)\, d\mu_j(w) 
 \\
& =\int_{[0,x]}  J_{k}^B\bigl(iy,2\sqrt{w}\,\bigr)\,dm_x(w).
\notag
\end{align*}
This readily implies the assertion. 
\end{proof}

\begin{remark} Our results on complete Sonine formulas for $J_k^B$ in Theorems \ref{int-rep-classical-chamber} and \ref{existence-pos-int-rep-lag} do not cover the case $k=(k_1, k_2)$ with  arbitrary $k_1 \geq 0, k_2 >0$ and $h >k_2(n-1).$ We conjecture that in this case, a positive Sonine formula exists as well. 
	
\end{remark}

We conclude this section with an immediate consequence of 
Theorem \ref{main_Bessel}, Lemma \ref{laguerre-to-bessel}, and the proof of Theorem \ref{existence-pos-int-rep-lag}:
   
\begin{corollary}\label{main_Laguerre} Let $a\geq -1/2,\, \, \alpha >0$  and $h>0.$  Assume that
for each partition $\kappa\in \Lambda_n^+$ there exist nonnegative connection coefficients $c_{\kappa,\lambda}\ge0$ such that
$$L_{\kappa}^{a+h}(x;\alpha)= 
\sum_{\lambda\subseteq\kappa}c_{\kappa,\lambda} L_{\lambda}^{a}(x;\alpha) \quad\quad  (x\in \mathbb R^n).$$
Then  $h$ is contained in the set
$ \,\Sigma(\alpha^{-1})$ of Theorem \ref{main_Bessel}.
\end{corollary}

\section{The limit $h\to\infty$}\label{infty}

We finally turn to some application of Theorem \ref{existence-pos-int-rep-lag} for $h\to\infty$ which is based on the fact that
the Bessel function $ J_{(k_1,k_2)}^B$ of type $B_n$ tends to a Bessel function of type $A_{n-1}$ as $k_1\to \infty.$ 

Indeed, the following limit relation follows easily from a comparison of the coefficients in \eqref{Bessel_A} and 
 \eqref{def-hypergeometric-function} and is well-known;
 see e.g.~\cite{RV1} where also estimations for the rate of convergence are given.

\begin{lemma}\label{b-to-a}
Let $n\geq 2$ and $k_2> 0$. Then, locally uniformly in $x,y\in\mathbb R^n$,
$$\lim_{k_1\to+ \infty} J_{(k_1,k_2)}^B(2\sqrt{k_1}\,x, iy)= J_{k_2}^A(x^2,-y^2).$$
 \end{lemma}

Using Prohorov's theorem 
as in the proof of Theorem  \ref{existence-pos-int-rep-lag},
 we obtain the following  integral representation from Lemma \ref{b-to-a} and  Theorem  \ref{existence-pos-int-rep-lag}:

\begin{theorem} \label{intrep-a-to-b}
Let $k=(k_1,k_2)$ with $k_1,k_2>0$. Then
  for each $x\in C_n^B$ there exists a unique probability measure $\mu_x = \mu_x(k) \in M^1(  C_n^B)$ such that
\begin{equation}\label{Sonine_limit-allg}
 J_{k_2}^A(x^2,-y^2) =
 \int_{ C_n^B}  J_{(k_1,k_2)}^B(\xi,iy)\, d\mu_x(\xi) \quad\text{for all } \, y\in\mathbb R^n .\end{equation}
\end{theorem}

\begin{proof} The uniqueness again follows from the injectivity of the Dunkl transform of measures. For the existence, fix $k_1,k_2$ as well as $x\in C_n^B$. Theorem \ref{existence-pos-int-rep-lag} 
shows that for each $j\in\mathbb N$ there is a probability measure $\mu_j\in M^1(C_n^B)$ such that
\begin{equation}\label{int-rep-j}
 J_{(k_1+jk_2,k_2)}^B\bigl(2\sqrt{k_1+jk_2}\cdot x\,, iy\bigr)\,=
\int_{ C_n^B}  J_{(k_1,k_2)}^B(\xi,iy) \,d\mu_j(\xi) \quad \forall \,y\in\mathbb R^n,
\end{equation}
where the support of $\mu_j$ is contained in  $\,\big[0,\,2\sqrt{k_1+jk_2}\cdot x\big]$. We prove that the sequence $(\mu_j)$ 
is tight, which implies that it has a subsequence which converges weakly to some probability measure $\mu_x\in M^1(C_n^B)$; 
for tightness and the existence of convergent subsequences we refer to Section 4  of \cite{Kal}. The arguments at the end of the proof of 
 Theorem  \ref{existence-pos-int-rep-lag} in combination with Lemma \ref{b-to-a} will then lead to
$$  J_{k_2}^A(x^2, -y^2) =
 \int_{ C_n^B}  J_{(k_1,k_2)}^B(\xi,iy)\, d\mu_x(\xi)$$
 as claimed.
In order to check the tightness of $(\mu_j)$, we recapitulate from formulas \eqref{def-hypergeometric-function} and  \eqref{ident_Bessel_B}
that for $k=(k_1,k_2),$
\begin{equation}
 J_{k}^B(z,w) = \,\sum_{\lambda\in \Lambda_n^+} \frac{1}{[\mu(k)]_\lambda^\alpha \,|\lambda|!\,4^{|\lambda|}} \cdot \frac{C_\lambda^\alpha(z^2) 
C_\lambda^\alpha(w^2)}{C_\lambda^\alpha({\bf 1})} \quad (z,w\in \mathbb C^n)
\end{equation}
with $\mu(k)=k_1+k_2(n-1)+\frac{1}{2}$. Hence, by (\ref{int-rep-j}), 
\begin{align}
\sum_{\lambda\in \Lambda_n^+} &\frac{(k_1+jk_2)^{|\lambda|}}{[\mu(k_1+jk_2,k_2)]_\lambda^\alpha\, |\lambda|!}\cdot 
 \frac{C_\lambda^\alpha(x^2)\, 
C_\lambda^\alpha(-y^2)}{C_\lambda^\alpha({\bf 1})}\notag\\
=& \sum_{\lambda\in \Lambda_n^+} \frac{1}{[\mu(k)]_\lambda^\alpha \,|\lambda|!\,4^{|\lambda|}} \cdot 
\frac{C_\lambda^\alpha(-y^2)}{C_\lambda^\alpha({\bf 1})} \cdot \int_{C_n^B} C_\lambda^\alpha(\xi^2) \> d\mu_j(\xi)\notag
\end{align}
for all $y\in [0,\infty[^n$. If we compare the coefficients of this power series in $y$ for $\lambda=(1,0,\ldots,0)$
and use that $C_{(1,0,\ldots,0)}^\alpha(z)=z_1+\ldots+z_n\,$ (c.f. the normalization \eqref{normal}),
we obtain from (\ref{general_Pochhammer}) and a straightforward computation that
\begin{equation}\label{second-moment}
\int_{C_n^B}(\xi_1^2+\ldots+\xi_n^2)\,d\mu_j(\xi)\,= \,\frac{4(k_1+jk_2)(k_1+k_2(n-1)+1/2)}{k_1+k_2(n+j-1)+1/2}
\end{equation}
which remains bounded as $j\to\infty$. By Exercise 4 in Section 4 of \cite{Kal}  this implies the tightness of $(\mu_j)$ and thus the claim.
\end{proof}

We  have no general explicit formula for the measures $\mu_\eta$ in Theorem \ref{intrep-a-to-b}.
 However, in certain cases explicit formulas are known. 
 For example, Lemma \ref{b-to-a} and Theorem \ref{int-rep-classical-chamber} lead for $k_2=d/2\in \{1/2,1,2\}$ to the following

\begin{corollary} 
Let $\,\mathbb F\in \{\mathbb R, \mathbb C, \mathbb H\}$ and $d = \text{dim}_{\mathbb R}\mathbb F$.
Then for $k_1\ge0$, $x\in C_n^B$ and all $z\in \mathbb C^n$,
\begin{equation}\label{Sonine_limit-2}
 J_{d/2}^A(x^2, -z^2) = \int_{ C_n^B} \int_{U_n(\mathbb F)} J_{(k_1,d/2)}^B
\bigl(\sqrt{\sigma(x u\xi u^{-1}x)}\,,z\bigr)\,du\,
   d\mu_{k_1,d}(\xi) \end{equation}
with the probability measure
$$ d\mu_{k_1,d}(\xi)=c_{k_1,d}\prod_{i=1}^n \xi_i^{k_1-1/2}  \prod_{i<j} (\xi_i-\xi_j)^{d}\cdot
e^{-(\xi_1+\ldots+\xi_n)/2} \> d\xi$$
with suitable normalizing constant $c_{k_1,d}>0$.
\end{corollary}

\end{document}